\documentclass[a4paper, 10pt, twoside, notitlepage]{amsart}

\usepackage{amsmath,amscd}
\usepackage{amssymb}
\usepackage{amsthm}

\usepackage{comment}
\usepackage{graphicx}
\usepackage{float}

\usepackage{mathrsfs}

\usepackage[ocgcolorlinks, linkcolor=blue]{hyperref}

\usepackage{amsmath,amsfonts,amsthm,mathrsfs,amssymb,cite}
\usepackage[usenames]{color}

\theoremstyle{plain}
\newtheorem{thm}{Theorem}[section]

\newtheorem{lem}{Lemma}[section]

\theoremstyle{definition}

\newtheorem{hypo}{Hypothesis}[section]
\theoremstyle{remark}

\newtheorem{rem}{Remark}[section]
\numberwithin{equation}{section}

\newcommand{\Om}{\Omega}

\newcommand{\R}{{\mathbb R}}
\newcommand{\C}{{\mathbb C}}

\newcommand{\D}{\partial}
\newcommand{\dA}{\delta A_0}

\title[Reconstruction for anisotropic inverse conductivity problems]{A direct reconstruction algorithm for the anisotropic inverse conductivity problem based on Calder\'on's method in the plane}

\author[R. Murthy]{Rashmi Murthy}
\address{Department of Mathematics, University of Helsinki, Finland}
\email{rashmi.murthy@helsinki.fi}

\author[Y.-H. Lin]{Yi-Hsuan Lin}
\address{Department of Applied Mathematics, National Chiao Tung University, Hsinchu, Taiwan}
\curraddr{}
\email{yihsuanlin3@gmail.com}

\author[K. Shin]{Kwancheol Shin}
\address{Department of Mathematics, Colorado State University, USA}
\email{kcshin@rams.colostate.edu}

\author[J. Mueller]{Jennifer L. Mueller}
\address{Department of Mathematics and School of Biomedical Engineering, Colorado State University, USA}
\email{mueller@math.colostate.edu}

\begin{document}
	
\begin{abstract} A direct reconstruction algorithm based on Calder\'on's linearization method for the reconstruction of isotropic conductivities is proposed for anisotropic conductivities in two-dimensions. To overcome the non-uniqueness of the anisotropic inverse conductivity problem, the entries of the unperturbed anisotropic tensors are assumed known \emph{a priori}, and it remains to reconstruct the multiplicative scalar field. The quasi-conformal map in the plane facilitates the Calder\'on-based approach for anisotropic conductivities. The method is demonstrated on discontinuous radially symmetric conductivities of high and low contrast.



\end{abstract}
	
	\maketitle

	\tableofcontents

\section{Introduction} \label{sec:intro}


The inverse conductivity problem was first proposed in Calder\'on's pioneer work \cite{calderon2006inverse}, in which the existence of a unique solution and a direct method of reconstruction of isotropic conductivities from the associated boundary measurements was given for the linearized problem in a bounded domain. This work inspired a large body of research on the global uniqueness question for the inverse conductivity problem and methods of reconstruction.  Calder\'on made use of special exponentially growing functions known as \emph{complex geometrical optics} solutions (CGOs), and these have proved useful for global uniqueness results for the inverse conductivity problem, leading also to a family of direct reconstruction algorithms known as D-bar methods. The reader is referred to \cite{Mueller2020, MuellerBook} and the references therein for further reading on D-bar methods.  Their relationship to Calder\'on's method was investigated in  \cite{knudsen2008reconstructions}.

The inverse conductivity problem is the mathematical model behind a medical imaging technology known as  \emph{electrical impedance tomography} (EIT). EIT has the attributes of being inexpensive, non-invasive, non-ionizing and portable.  Medical applications of EIT include pulmonary imaging \cite{panoutsos2007electrical,frerichs2009assessment,victorino2004imbalances,Martins2019}, breast cancer detection \cite{1344192,cherepenin20013d,kao20063d,kerner2002electrical,zou2003review}, human head imaging \cite{yerworth2003electrical,agnelli2020classification,Boverman, Malone2014}, and others.  See also the review articles \cite{lionheart2004eit,cheney1999electrical,hanke2003recent,brown2003electrical} on EIT.
The literature of reconstruction algorithms for the isotropic conductivity is extensive.  Implementations of Calder\'on's method for isotropic conductivities include \cite{bikowski20082d, Muller1, Peter_2017}.   However, in reality, many tissues in the human body are anisotropic, meaning that electrical current will be conducted in spatially preferred directions.
Anisotropic conductivity distributions are prevalent in the body, with significant differences in the transverse and lateral directions in the bone, skeletal and cardiac muscle, and in the brain \cite{abascal2008use, barber1990quantification}. 
However, medical EIT imaging typically neglects anisotropic properties of the conductivity and reconstructs an isotropic approximation.  This can lead to artifacts in the resulting images and incorrect estimates of conductivity values.

In this work, we want to study the anisotropic conductivity equation 
\begin{align}\label{anisotropic conductivity equation in Sec 1}
	L_A u :=\nabla \cdot (A(x)\nabla u)=\sum_{j,k=1}^2 \D_{x_k}\left( A^{jk}(x)\D_{x_j}u\right)=0 \quad \text{ in }\Omega,
\end{align}
where $\Omega\subset \R^2$ is a bounded Lipschitz domain and $A(x)=(A^{jk}(x))_{1\leq j,k\leq 2}$ is a positive definite symmetric matrix. The explicit regularity assumptions of $A(x)$ will be characterized later (see Hypothesis \ref{hypothesis of A}). In order to study inverse problems in anisotropic media, the key step is to transform the anisotropic conductivity equation \eqref{anisotropic conductivity equation in Sec 1} into an isotropic one, by using the \emph{isothermal coordinates}, which are related to \emph{quasi-conformal maps} \cite{ahlfors2006lectures}. This method is widely used in the study of the Calder\'on type inverse problems in the plane.

Contrary to the isotropic case, knowledge of the \emph{Dirichlet-to-Neumann} (DN) map is not sufficient to recover an anisotropic conductivity   \cite{greenleaf2003anisotropic, kohn1983identification}. The non-uniqueness of the anisotropic problem stems from the fact that any diffeomorphism of $\Omega$ which keeps the boundary points fixed has the property of leaving the DN map unchanged (i.e., one can find a diffeomorphism with the boundary map being identity), despite the change in conductivity \cite{astala2006calderon}. 
However, a uniqueness result was proved in \cite{lionheart1997conformal} under the assumption that the conductivity is known up to a multiplicative scalar field. The uniqueness and stability results for an anisotropic conductivity of the form $ A = A(x,a(x)) $, where $a(x)$ is an unknown scalar function was proved in \cite{alessandrini2001determining}. In the same paper, the uniqueness result for the interior conductivities of $A$ were also proved by piecewise analytic perturbations of scalar term $a$. {\em A priori} knowledge of the preferred directions of conductivity, or the entries of the tensor of anisotropy, may be obtained, for example, from a diffusion tensor MRI, as discussed in 
\cite{abascal2011electrical}.

In this work, we provide a direct reconstruction algorithm under the assumption that the anisotropic conductivity is a ``small perturbation" (under the matrix notion) of a given $2\times 2$ positive definite matrix with constant entries in the plane.
Other approaches to the reconstruction problem for anisotropic conductivities include \cite{Breckon,Glidewell2D, Glidewell3D,abascal2008use,HamiltonReyes2016, Lionheart_2010}. 
The quasi-conformal map in the plane and \emph{invariance} of the equations under a change of coordinates is the key to the algorithm. But, we note that in three spatial dimensions, there is no quasi-conformal map, and so our approach is not applicable.

CGO solutions have also been used to solve the inverse obstacle problem under various mathematical models.  For the isotropic case, see, for example, \cite{KS2014,KSU2011,nakamura2007identification,SW2006,SY2012reconstruction,UW2008}. For the anisotropic case in the three-dimensions, one can use the \emph{oscillating-decaying solutions} to reconstruct unknown obstacles in a given medium, see \cite{KLS2015enclosure,lin2014reconstruction,NUW2005(ODS),NUW2006}. 
It is worth mentioning that for the anisotropic case in the plane, one can also construct CGOs via the quasi-conformal map, and we refer readers to \cite{takuwa2008complex} for more detailed discussions.

The paper is organized as follows. Section \ref{sec:model} contains the mathematical formulation of the anisotropic inverse conductivity problem. In Section \ref{sec:2}, we provide a rigorous mathematical analysis for the anisotropic elliptic problem. We prove that the linearization of the quadratic form is injective when evaluated at any positive definite $2\times 2$ constant matrix. The tool is to use the quasi-conformal map in the plane. In Section \ref{sec:3}, we provide a reconstruction methods based on Calder\'on's approach for anisotropic conductivities. In addition, we also provide a numerical implementation for the inverse anisotropic conductivity problem in Section \ref{sec:4}.

\section{Mathematical formulation} \label{sec:model}

The mathematical model for the EIT problem with an anisotropic conductivity can be formulated as follows:
Let $\Omega \subset \R^2$ be a simply connected domain with Lipschitz boundary $\partial \Omega$. Assume that the following conditions hold.
\begin{hypo}\label{hypothesis of A}
	Let $ A(x) = \left(A^{jk}(x)\right)_{1\leq j,k\leq 2}$ be the anisotropic conductivity, which satisfies:
	\begin{itemize}
		\item[(a)] Symmetry: $A^{jk}(x)=A^{kj}(x)$ for $x\in \Omega$.
		
		\item[(b)] Ellipticity: There is a universal constant $\lambda>0$ such that 
		
		$$
		\sum_{j,k=1}^2 A^{jk}(x)\xi_j \xi _k \geq \lambda |\xi|^2, \text{ for any }x\in \Omega \text{ and }\xi=(\xi_1,\xi_2) \in \R^2.
		$$
		
		\item[(c)] Smoothness: The anisotropic conductivity $A\in C^1(\overline\Omega;\R^{2\times 2})$.
		
	\end{itemize}
\end{hypo}
Let $ \phi \in H^{1/2}(\D \Omega)$ be the voltage given on the boundary. The electric field $u$ arising from the applied voltage $\phi$ on $\partial\Omega$ is governed by the following second order elliptic equation
\begin{align}\label{anisogenLap}
\begin{cases}
\nabla \cdot (A \nabla u) = \sum_{j,k = 1}^{2} \frac{\D}{\D x^j}\left( A^{jk} (x) \frac{\D}{\D x^k}u \right)= 0 &\text{ in } \Omega, \\
u= \phi  &\text{ on } \D \Omega.
\end{cases}
\end{align}
Under Hypothesis \ref{hypothesis of A} for $A $, given any Dirichlet data $\phi$ on $\D \Omega$, it is known that \eqref{anisogenLap} is well-posed (for example, see \cite{gilbarg2015elliptic}). Therefore, the DN map is well-defined and is given by
\begin{align*}
\Lambda_A:H^{1/2}(\D\Omega)&\to H^{-1/2}(\D\Omega)\\
\Lambda_{A}: \phi &\mapsto  \nu \cdot A\nabla u|_{\D \Omega}= \left.\sum_{j,k =1}^{2} A^{jk}(x) \frac{\D u}{\D x^j}\nu_k\right|_{\D\Omega} ,
\end{align*} 
where
$u \in H^{1} (\Omega)$ is the solution of \eqref{anisogenLap}, and $\nu=(\nu_1,\nu_2)$ is the unit outer normal vector on $\D \Omega$. It is natural to consider the quadratic form $Q_{A, \Omega}(\phi)$ with respect to \eqref{anisogenLap}, which is defined by 
\begin{equation}\label{poweraniso}
Q_{A,\Omega}(\phi) := \int_{\Omega} \sum_{j,k =1}^{2} A^{jk}(x) \frac{\D u}{\D x^j} \frac{\D u}{\D x^k}\, dx = \int_{\D \Omega} \Lambda_{A} (\phi) \phi \, dS,
\end{equation}
where $dS$ denotes the arc length on $\D \Omega $ and we have utilized the integration by parts in the last equality. The quantity $Q_{A, \Omega}$ represents the power needed to maintain the potential $\phi$ on $\D \Omega$. By the symmetry of the matrix $A $, knowledge of $Q_{A, \Omega}$ is equivalent to knowing $\Lambda_{A}$. 

The inverse problem of anisotropic EIT is to ask whether $A$ is uniquely determined by the quadratic form $Q_{A,\Omega}$ and if so, how to calculate the matrix-valued function $A$ in terms of $Q_{A,\Omega}$. 

\section{Injectivity of the inverse problem} \label{sec:2}

Assume the conductivity $A(x)$ satisfies Hypothesis \ref{hypothesis of A} and 
$A(x)$ is of the form 
$$A = a(x)A_0, $$ 
where $a(x)$ is a scalar function to be determined and $A_0$ is a known \emph{constant} $2\times 2$ symmetric positive definite anisotropic tensor.  Then by \cite{lionheart1997conformal}, the inverse anisotropic conductivity problem of determining $a(x)$ has a unique solution.
		
Introduce the norms in the space of $A$ and in the space of quadratic forms $Q_{A}\equiv Q_{A,\Omega}(\phi)$ as follows. Let 
\begin{equation}\label{norm of phi}
\| \phi \|^{2} := \int_{\Omega} |\nabla u|^2 dx,
\end{equation}
where $u\in H^1(\Omega)$ is the solution of the following second order elliptic equation with constant matrix-valued coefficient,  
\begin{align}\label{equ u}
\begin{cases}
\nabla \cdot (A_0 \nabla u) =  0 &\text{ in }\Omega,\\
u= \phi  & \text{ on }\D\Omega, 
\end{cases}
\end{align}
and 
\begin{equation}
\|Q_A \| = \sup_{ {\| \phi \| \leq 1}} \left|Q_A(\phi)\right|.
\end{equation}

In the spirit of Calder\'on \cite{calderon2006inverse}, the next step is to show that  the mapping from conductivity to power, namely 
\begin{equation} \label{phi}
\Phi : A \rightarrow Q_{A},
\end{equation} 
is analytic, where the conductivity to power map $Q_{A} $ is given by \eqref{poweraniso}.   The argument outlined below establishes that $\Phi$ is analytic.

Consider  $A(x)$ as a perturbation from the constant matrix $A_0$ of the form
\begin{align}\label{definition of conductivity}
A(x): =(1+\delta(x))A_0,
\end{align}  
where $\delta(x)$ is regarded as a \emph{scalar} perturbation function. We further assume that $\|\delta\|_{L^\infty(\Omega)} <1$ is sufficiently small so that the matrix $A$ is also positive definite in order to derive the well-posedness of the following second order elliptic equation 
\begin{align}\label{equ w}
\begin{cases}
L_{A}(w):= \nabla \cdot (A(x) \nabla w) = 0 & \text{ in }\Omega, \\
 w= \phi & \text{ on } \D \Omega.
\end{cases}
\end{align}
As in  Calder\'on \cite{calderon2006inverse} for the isotropic case, we will use perturbation arguments. Let $w :=u+v $, where $w$ is the solution of \eqref{equ w}, and $u$ is the solution of \eqref{equ u} with the same boundary data $w=u=\phi$ on $\D \Omega$. Then 
\begin{equation}\label{eqn13}
L_A (w) = L_{(1+\delta(x))A_0} (u+v) = L_{A_0} v+L_{\dA}v + L_{\dA}u = 0 \quad \text{ in }\Omega,
\end{equation} 
and  $v\vert_{\partial \Omega} = 0$. Then we have the following estimate for the function $v$.

\begin{lem}
The operator  $L_{A_0} $ has a bounded inverse operator $G$ and $v$ has the following $H^1$ bound
\begin{equation} \label{v_bound}
\|v\|_{H^1(\Omega)} \leq \frac{\|G\|_{\mathcal L(H^{-1};H^1)} \|\delta\|_{\infty}\|A_0\|_F \|\phi\|}{1-\|G\| \|\delta\|_{\infty}\|A_0\| },
\end{equation}
where $\|G\|_{\mathcal L(H^{-1};H^1)}$ denotes the operator norm from $H^{-1}$ to $H^1$, $\|A_0\|_F$ stands for the Frobenius norm of the matrix $A_0$ and $\|\phi\|$ is given by \eqref{norm of phi}.
\end{lem}

\begin{proof}
Since $u\in H^1(\Omega)$ is the unique solution of the boundary value problem $L_{A_0} u = 0$ in $\Omega$, $ u = \phi$ on $\D \Omega$, the operator $L_{A_0}$ has a bounded inverse $G$.

Then from \eqref{eqn13} one sees,
\begin{equation}\label{relation 1}
G\left(L_{A_0}v +L_{\delta A_0} v +L_{\delta A_0} u \right)= 0.
\end{equation}
That is, 
\begin{equation}\label{representatiln of v}
\left(I+G L_{\delta A_0}\right)v = - G L_{\delta A_0} u.
\end{equation}
Note that 
\begin{equation}\label{nL}
\| L_{\delta A_0} w \|_{H^{-1}(\Omega)}= \sup_{\psi \in H^{1}_{0} (\Omega)} \frac{\left|\int_{\Omega} \nabla \cdot [(\delta(x) A_0) \nabla w] \psi \, dx\right|}{\| \psi\|_{H^1_0(\Omega)}},
\end{equation}
and 
\begin{align}\label{inequality 1}
\begin{split}
\left|\int_{\Omega} \nabla \cdot [ (\delta(x) A_0) \nabla w]\psi \, dx \right| = &\left| \int_{\Omega} \nabla \psi \cdot [(\delta(x) A_0) \nabla w] \, dx \right| \\
\leq& \| \delta(x) A_0 \|_{L^\infty(\Omega)}   \|\nabla w \|_{L^2(\Omega)}\|\psi\|_{H^1_0(\Omega)}.
\end{split}
\end{align}
Thus, from \eqref{nL} and \eqref{inequality 1}, 
\begin{equation}\label{relation 2}
\| L_{\delta A_0} w \|_{H^{-1}(\Omega)} \leq\| \delta(x) A_0 \|_{L^\infty(\Omega)}   \|\nabla w \|_{L^2(\Omega)}.
\end{equation}
Next, consider the operator norm 
\begin{align}
\|G L_{\delta A_0} \|_{\mathcal{L}(H^1;H^1)} = \sup_{w \neq 0} \frac{\| GL_{\delta A_0} w\|_{H^1(\Omega)}}{\|w\|_{H^1(\Omega)}}\leq \frac {\|G\|_{\mathcal{H^{-1};H^1}} \|L_{\delta(x) A_0} w \|_{H^{-1}(\Omega)}}{\| w \|_{H^1(\Omega)}}.
\end{align}
That is, 
\begin{equation}
\|G L_{\delta A_0} \|_{\mathcal L(H^1;H^1)} \leq \|G\|_{\mathcal L(H^{-1};H^1)} \| \delta A_0 \|_{L^\infty(\Omega)},
\end{equation}
where $\|\cdot \|_{\mathcal L(X;Y)}$ stands for the operator norm from the Banach space $X$ to $Y$.
Moreover, when  
$$\| \delta(x) A_0 \|_{L^\infty(\Omega)} < \dfrac{1}{\|G\|_{\mathcal L(H^{-1};H^1)}} ,$$
the Neumann series 
$$ 
\left[\sum_{j = 0}^{\infty} (-1)^j(GL_{\delta A_0})^j\right] \left(GL_{\delta A_0}u\right),
$$
converges, and from \eqref{representatiln of v}, one has 
\begin{equation}
v = - \left[ \sum_{j = 0}^{\infty} (-1)^j (GL_{\delta A_0})^j \right](GL_{\delta A_0} u).
\end{equation} 
It is easy to see that 
\begin{align}\label{relation 3}
\notag\|v+GL_{\delta A_0} v \|_{H^1(\Omega)} \geq & \|v\|_{H^1(\Omega)} - \|GL_{\delta A_0} v \|_{H^1(\Omega)} \\
\geq  &\|v\|_{H^1(\Omega)}\left(1- \|G\|_{\mathcal L(H^{-1};H^1)} \|\delta(x) A_0\|_{L^\infty(\Omega)}\right).
\end{align}  
Finally, by using relations \eqref{relation 1} and \eqref{relation 2}, 
\begin{align*}
&\quad \left(1-\|G\|_{\mathcal L(H^{-1};H^1)} \| \delta(x) A_0\|_{L^\infty(\Omega)}\right) \|v\|_{H^1(\Omega)} \\
&\leq \|v + G (-L_{\delta A_0}u - L_{A_0} v)\|_{H^1(\Omega)} \\
&= \|v - G L_{\delta A_0} u -v\|_{H^1(\Omega)}\\
&\leq \|G\|_{\mathcal L(H^{-1};H^1)} \|L_{\delta A_0} u \|_{H^{-1}(\Omega)}\\
&= \|G\|_{\mathcal L(H^{-1};H^1)} \| \delta(x) A_0 \|_{L^\infty(\Omega)} \| \phi\|,
\end{align*}
and substituting \eqref{relation 3} into the above inequality results in
\begin{equation}\label{estimate for v}
\|v \|_{H^1(\Omega)} \leq \frac{\|G\|_{\mathcal L(H^{-1};H^1)} \| \delta(x) A_0 \|_{L^\infty(\Omega)} \| \phi \|}{1- \|G\|_{\mathcal L(H^{-1};H^1)} \| \delta(x) A_0\|_{L^\infty(\Omega)}}.
\end{equation}
Therefore, the above calculations allow us to conclude that the mapping  $\Phi $ defined by \eqref{phi} is analytic at $A_{0}$, which completes the proof.
\end{proof}

Next,  let us linearize the map $Q_{A}(\phi) $ around a positive definite constant matrix $A(x) = A_0 $ as follows: 
\begin{align}
\notag Q_{(1+ \delta(x)) A_0}(\phi) = & \int_{\Omega} \left[(1 + \delta(x))A_0 \nabla w\right] \cdot \nabla w \, dx \\
\notag=& \int_{\Omega} \left[(A_0 +\delta(x) A_0)\nabla u\right] \cdot \nabla u \, dx + 2\int_{\Omega} (\delta(x) A_0\nabla u) \cdot \nabla v \, dx\\
&+\int_\Omega \left[(A_0+\delta(x) A_0)\nabla v\right]\cdot \nabla v \, dx,
\end{align}
where we have used that $\nabla \cdot (A_0\nabla u)=0$ in $\Omega$.
We now show that 
$$ 
\int_{\Omega} (\delta(x) A_0 \nabla u) \cdot \nabla v \, dx \quad \text{ and } \quad \int_\Omega \left[(A_0+\delta(x) A_0)\nabla v\right]\cdot \nabla v \, dx
$$ 
are of $\mathcal{O}(\| \delta  \|) $. It is easy to see that 
\begin{align}\label{inequality v 1}
\left|\int_\Omega (\delta(x) A_0)\nabla u\cdot \nabla v\, dx\right|\leq C_{A_0} \|\delta\|_{L^\infty(\Omega)}\|\phi \|\|v\|_{H^1(\Omega)},
\end{align}
and 
\begin{align}\label{ineqaulity v 2}
\left|\int_\Omega \left[(A_0+\delta(x) A_0)\nabla v\right]\cdot \nabla v \, dx\right|\leq C_{A_0}\left(1+\|\delta\|_{L^\infty(\Omega)}\right)\|v\|_{H^1(\Omega)}^2,
\end{align}
for some constant $C_{A_0}>0$ independent of $\delta$. By inserting \eqref{estimate for v} into \eqref{inequality v 1}, \eqref{ineqaulity v 2} and taking $\|\delta\|_{L^\infty (\Omega)}$ sufficiently small, one obtains the desired result.
Thus, the Fr\'echet derivative of the quadratic form $Q_A(\phi)$  at $A(x)=A_0$ is given by 
\begin{equation}\label{dQ1}
dQ_{A}(\phi)\Big\vert_{A = A_{0}} = \int_\Omega \left((\delta(x) A_0 \right) \nabla u) \cdot \nabla u \, dx, 
\end{equation} 
where $u\in H^1(\Omega)$ is a solution of $\nabla \cdot (A_0\nabla u)=0$ in $\Omega$ with $u=\phi$ on $\partial \Omega$. 

\begin{thm}
The Fr\'echet derivative $ \left. dQ_A(\phi)\right\vert_{A = A_{0}}$ is injective.
\end{thm}
\begin{proof}

In order to prove that $\left. dQ_A(\phi)\right|_{A = A_{0}}$ is injective, we only need to show that 
\begin{align*}
\int_\Omega \left(\delta(x) A_0\nabla u \right) \cdot \nabla u\,dx=0 \quad  \text{ implies that }\quad \delta \equiv 0,
\end{align*}
where $u \in H^1(\Omega)$ is a solution of $L_{A_0}u=0$ in $\Omega$ with $u=\phi$ on $\D \Om$. On the other hand, since the last integral in \eqref{dQ1} vanishes for all such $u$, then it is equivalent to prove 
\begin{equation}\label{inj}
\int_\Omega ((\delta(x) A_0) \nabla u_1)\cdot \nabla u_2\ dx = 0,
\end{equation} 
where $u_1$, $u_2\in H^1(\Omega)$ are solutions of $L_{A_0}u_1 = L_{A_0}u_2 = 0$ in $\Omega$.  Inspired by \cite{calderon2006inverse}, we want to find special exponential solutions to prove it.

In order to achieve our aim, we utilize the celebrated quasi-conformal map in the plane. We first identify $\R^2$ with the complex plane $\C$. Recall that $\delta(x)\in  C^1(\overline{\Omega})$ since the conductivity $A(x)$ is a $C^1(\overline{\Omega})$ matrix function. We now extend $\delta(x)$ to $\C$ (still denoted by $\delta(x)$) with $\delta(x)\in C^1_0 (\C)$ by considering $\delta(x)\equiv 0$ for $|x|>r$, for some large constant $r>0$ such that $\Omega \Subset B_r(0)$. Meanwhile, we also extend the constant matrix $A_0$ to $\C$, denoted by $A_0(x)$, such that $A_0(x)\in C^1(\C)$ with $A_0(x)=I_2$ (a $2\times2$ identity matrix) for $|x|>r$.

From \cite[Lemma 3.1]{astala2005calderons} and \cite[Theorem 2.1]{takuwa2008complex}, given a $C^1$-smooth anisotropic conductivity $A(x)$, it is known that there exists a $C^1$ bijective map $\Phi^{(A)}:\mathbb R^2 \to \mathbb R^2$ with $y=\Phi^{(A)}(x)$ such that 
\begin{align}\label{scalar conductivity}
\Phi^{(A)}_*A = \left(\det A\circ (\Phi^{(A)})^{-1}\right)^{1/2}I_2,
\end{align}
is a scalar conductivity, where  
\begin{align}\label{transformation_formula}
\Phi^{(A)}_*A(y)=\left. \dfrac{\nabla \Phi^{(A)}(x)A(x)\nabla (\Phi^{(A)})^T(x)}{\det (\nabla \Phi^{(A)}(x))}\right|_{x=(\Phi^{(A)})^{-1}(y)}, 
\end{align}
with 
\begin{align*}
&(\Phi^{(A)}_* A)^{i\ell}(y) \\
:=&\dfrac{1}{\det (\nabla \Phi^{(A)})}\sum _{j,k=1}^2\D_{x^j}(\Phi^{(A)})^i(x) \D_{x^k}(\Phi^{(A)})^\ell(x) A^{jk}(x)\Big\vert_{x=(\Phi^{(A)})^{-1}(y)}.
\end{align*}
Moreover, $\Phi^{(A)}$  solves the following Beltrami equation in the complex plane $\mathbb C$, 
\begin{align} \label{Belt}
\overline{\D}\Phi^{(A)}=\mu_A \D \Phi^{(A)},
\end{align}
where 
\begin{align}\label{muA coeff}
\mu _A=\dfrac{A^{22}-A^{11}-2iA^{12}}{A^{11}+A^{22}+2\sqrt{\det A}},
\end{align}
and 
\begin{align*}
\overline{\D}=\dfrac{1}{2}(\D_{x_1}+i\D_{x_2}),\quad \D=\dfrac{1}{2}(\D_{x_1}-i\D_{x_2}).
\end{align*}
Here we point out that the coefficient $\mu_A$ defined in \eqref{muA coeff} is supported in $B_r\subset \C$, since $\delta(x)=0$ and $A_0$ was extended $C^1$-smoothly to the identity matrix $I_2$ for the same domain $|x|>r$.

Next, for the given anisotropic tensor $A_0$, we can find a corresponding Beltrami equation for $A_0$. We do the same extension of $A_0$ as before, such that $A_0(x)\in \C^1(\C)$ with $A_0(x)=I_2$ for $|x|>r$ (for the same large number $r>0$). Now, from the representation formulas \eqref{scalar conductivity} and \eqref{transformation_formula}, we know that there exists a quasi-conformal map $\Phi^{(A_0)}$ such that
\begin{align}\label{transformation formula 2}
\widetilde A_0:=&\Phi^{(A_0)}_* A_0=\left. \dfrac{\nabla \Phi^{(A_0)}(x)A_0\nabla (\Phi^{(A_0)})^T(x)}{\det (\nabla \Phi^{(A_0)}(x))}\right|_{x=(\Phi^{(A_0)})^{-1}(y)}=\sqrt{\det A_0}I_2 
\end{align}  
with $\det A_0>0$. Furthermore, $\Phi^{(A_0)}$ solves the Beltrami equation 
$$
\overline{\D}\Phi^{(A_0)}=\mu_{A_0} \D \Phi^{(A_0)},
$$
where 
\begin{align}\label{muA_0 coeff}
\mu _{A_0}=\dfrac{A_0^{22}-A^{11}-2iA_0^{12}}{A^{11}+A_0^{22}+2\sqrt{\det A_0}}.
\end{align}
Similarly, we also have that $\mu_{A_0}$ is supported in the same disc $B_r$ since $A_0=I_2$ for $|x|>r$. 
Note that $\delta=\delta(x)$ is a scalar function, then by using the formula \eqref{transformation_formula} and \eqref{transformation formula 2}, one can see that 
\begin{align*}
\Phi^{(A_0)}_* (\delta(x) A_0)=&\left. \dfrac{\nabla \Phi^{(A_0)}(x)\left(\delta(x)A_0\right)\nabla (\Phi^{(A_0)})^T(x)}{\det (\nabla \Phi^{(A_0)}(x))}\right|_{x=(\Phi^{(A_0)})^{-1}(y)} \\
=&\delta(x)|_{x=(\Phi^{(A_0)})^{-1}(y)}\sqrt{\det A_0}I_2.
\end{align*}
Let $\Phi\equiv\Phi^{(A_0)}$, $\widetilde \Om:=\Phi(\Omega)$ and $\widetilde u_j(y):=u_j\circ (\Phi^{-1}(y)$ for $j=1,2$, by using \eqref{inj} and change of variables $y=\Phi(x)$ via the quasi-conformal map, then we obtain that
\begin{align}\label{transf-inj}
\int_{\widetilde \Omega}\Phi_*(\delta(x) A_0) \nabla_y \widetilde u_1 \cdot \nabla_y \widetilde u_2 \, dy=\int_\Omega ((\delta(x) A_0) \nabla u_1)\cdot \nabla u_2 \, dx =0,
\end{align}
where $\widetilde u_j$ are solutions of 
\begin{align}\label{A0_equation}
L_{\widetilde A_0}\widetilde u_1=L_{\widetilde A_0}\widetilde u_2=0 \text{ in }\widetilde \Omega.
\end{align}


In fact, \eqref{A0_equation} is equivalent to the Laplace equation $\Delta _y\widetilde u_j=0$ in $\widetilde \Omega$ for $j=1,2$ because $\widetilde A_0=\sqrt{\det A_0}I_2$ with $\det A_0$ being a positive constant. Based on Calder\'on's constructions \cite{calderon2006inverse}, we can consider two special exponential solutions in the transformed space as follows. Let $\xi \in \mathbb R^2$ be an arbitrary vector and $b\in \mathbb R^2$ such that $\xi \cdot b=0$ and $|\xi |=|b|$, then one can define 
\begin{align}\label{exponential solution}
\widetilde u_1(y):=e^{\pi i(\xi \cdot y)+\pi (b\cdot y)}\quad\text{ and }\quad \widetilde u_2(y):=e^{\pi i(\xi \cdot y)-\pi (b\cdot y)},
\end{align}
and it is easy to check that $\widetilde u_1$ and $\widetilde u_2$ are solutions of Laplace's equation. By substituting these exponential solutions \eqref{exponential solution} into \eqref{transf-inj}, one has 
\begin{align*}
2\pi |\xi |^2\int _{\widetilde \Omega} \left(\delta\circ \Phi^{-1}(y)\right) \sqrt{\det A_0}\,e^{2\pi i \xi \cdot y}\, dy=0, \text{ for any }\xi \in \mathbb R^2,
\end{align*}
which implies that $\delta=0$, due to the positivity of $\det A_0$. This proves the assertion.
\end{proof}

\begin{rem}
It is worth mentioning that 
\begin{itemize}
\item[(a)] Due to the remarkable quasi-conformal mapping in the plane, one can reduce the anisotropic conductivity equation into an isotropic one. This method helps us to develop the reconstruction algorithm for the anisotropic conductivity equation proposed by Calder\'on \cite{calderon2006inverse}.

\item[(b)] The method fails when the space dimension $n\geq 3$, because there are no suitable exponential-type solutions for the anisotropic case. For the three-dimensional case, we do not have complex geometrical optics solutions but we have another exponential solution, which is called the oscillating-decaying solution (see \cite{lin2014reconstruction}). 
\end{itemize}
\end{rem}

\section{The linearized reconstruction method}\label{sec:3}

Since the tensor of anisotropy $A_0$ is known {\em a priori}, we can now transform the problem to the isotropic case, reconstruct the transformed conductivity on the transformed domain using Calder\'on's method on an arbitrary domain as in \cite{Peter_2017}, and then use the quasi-conformal map to transform the conductivity back to the original one.  Recall that from the definitions of our choices of extensions for $A_0$ and $\delta(x)$, the representation formulations of \eqref{muA coeff} and \eqref{muA_0 coeff} yield that 
$$
\mu_A = \mu_{A_0} \text{ in }\C, \quad \text{ and }\quad \mu_A=\mu_{A_0}=0 \text{ for }|x|>r>0.
$$
Thus, without loss of generality, we may assume that $\Phi^{(A_0)} = \Phi^{(A)}$ by using $\mu_{A_0}=\mu_A$. In the rest of this article, we simply denote the quasi-conformal map by $y=\Phi(x)\equiv\Phi^{(A_0)}(x) = \Phi^{(A)}(x)$.
By utilizing the change of variables via the quasi-conformal map, we also have that
$$
\int_{\partial \Omega} u_1 \left( \Lambda_A u_2 \right) \, dS = \int_{\partial \widetilde{\Omega}} \widetilde{u}_1 \left( \Lambda _{\widetilde{A}} \widetilde{u}_2\right)\, dS,
$$ 
where $\widetilde{\Omega}=\Phi(\Omega) $, $\widetilde{A}(y)=\Phi^{(A)}_*A(y) $, defined in  \eqref{transformation_formula}, and $\widetilde u_j = u_j\circ \Phi^{-1}(y)$ for $j=1,2$.

Since the DN data is preserved under the quasi-conformal map (i.e., change of variables), we can reconstruct the scalar conductivity $\tilde{a}(x)$ from Calder\'on's method as in \cite{Peter_2017}, and then map it back.
Defining Calder\'on's linearized bilinear form for the isotropic problem by 
\begin{eqnarray}
B(\phi_1,\phi_2) = \int_{\partial \Omega} w_1 \left(\Lambda _A w_2 \right)\, dS, 
\label{equB}
\end{eqnarray}
$w_1|_{\partial \Omega} = \phi_1 = u_1|_{\partial \Omega}$ and  $w_2|_{\partial \Omega} = \phi_2 = u_2|_{\partial \Omega}$,
this can be computed directly from our measured data.  Calder\'on proved \cite{calderon2006inverse} that the Fourier transform of an isotropic conductivity can be decomposed into two terms, one of which is negligible for small perturbations.  Thus, denoting the Fourier transform of a function $f$ by $\widehat{f}$, by \cite{calderon2006inverse} we can write
\begin{align}
	\widehat{\tilde{a}}(z)=\widehat{F}(z)+R(z)\label{gammahat},
\end{align}
where 
\begin{align}\label{gammaFandR}
\begin{split}
\widehat{\tilde{a}}(z) =&-\frac{1}{ 2 \pi^2 |z|^2} \int_{\tilde{\Omega}} (1+\tilde{\delta}(x)) \nabla u_1\cdot \nabla u_2 \, dx \\
= &-\frac{1}{ 2 \pi^2 |z|^2}\int_{\tilde{\Omega}} \tilde{a}(x)  e^{2\pi i(z\cdot x)}\,dx \\
\widehat{F}(z) =& -\frac{1}{2\pi^2|z|^2}B \left( e^{i\pi(z\cdot x)+\pi(b\cdot x)}, e^{i\pi(z\cdot x)-\pi(b\cdot x)}\right)
\\
R(z) =&\frac{1}{2\pi^2|z|^2} \int_{\tilde{\Omega}} \tilde{\delta} (\nabla u_1 \cdot \nabla v_2+ \nabla v_1 \cdot \nabla u_2) +(1+\tilde{\delta}) \nabla v_1 \cdot \nabla v_2 \, dx.
\end{split}
\end{align}
 For small $|z|$, the term $R(z)$ is small when the perturbation $\tilde{\delta}$ is small in magnitude and is to be neglected in numerical implementation.
Thus, the isotropic conductivity can be approximated by the inverse Fourier transform of $\widehat{F}(z)$. Since the deformed domain $\tilde{\Omega}$ is not circular, we adopt the algorithm introduced in \cite{Peter_2017}, in order to compute the function $\widehat{F}(z)$. 
We will first invert  $\widehat F(z)$
numerically to obtain a reconstruction of the scalar isotropic conductivity $\widetilde
A= \Phi_*{A(x)}$ on $\widetilde{\Omega}=\Phi^{(A_0)}\Omega $. 
Next, we compute the mapping $\Phi^{(A_0)}$ by solving the Beltrami equation \eqref{Belt}.  We then pull back the scalar conductivity from the deformed coordinates to the original coordinates by applying $(\Phi^{(A_0)})^{-1}$ to $\tilde{a}$ to obtain $a(x)$.   Finally, we obtain the anisotropic conductivity $A(x)$ by multiplying $a(x)$ by the known matrix $A_0$. 

%

\section{Numerical implementation}\label{sec:4}

\subsection{Numerical solution of the forward problem for data simulation}\label{sec_forward}

A finite element method (FEM) implementation of the complete electrode model (CEM) \cite{Somersalo} for EIT was developed for data simulation. We first provide the equations of the CEM.  Assume the anisotropic conductivity $A$ satisfies Hypothesis \ref{hypothesis of A}.
Then it satisfies the anisotropic generalized Laplace equation 
\begin{equation}
\nabla \cdot (A \nabla u) = \sum_{j,k = 1}^{2} \frac{\D}{\D x^j}\left( A^{jk} (x) \frac{\D u}{\D x^k} \right)= 0 \text{ in } \Omega.
\end{equation}
The boundary conditions for the CEM with $L$ electrodes are defined as follows.  The current $I_{l}$ on the $l$ the electrode is given by
\begin{equation}
 \int_{e_l} \sum_{j,k = 1}^{2} \left( A^{jk} (x) \frac{\D u}{\D x^k} \right) \, dS= I_{l},  \qquad  l = 1,2,...,L, 
\end{equation}
\begin{equation*}
\sum_{j,k = 1}^{2} \left( A^{jk} (x) \frac{\D u}{\D x^k} \right) = 0 \qquad \text{off} \qquad \bigcup_{l = 1}^{L} e_l,
\end{equation*}
where $e_l$ is the region covered by the $l$-th electrode, and $\nu$ is the outward normal to the surface of the body.
The voltage on the boundary is given by 
\begin{equation}
    u + z_{l} \sum_{j,k = 1}^{2} \left( A^{jk} (x) \frac{\D u}{\D x^k} \right) = U_l \qquad {\text{on}} \quad e_l \qquad \text{for} \quad l = 1,2,...,L,
 \end{equation}
 where $z_l$ is the {\em contact impedance} corresponding to the $l$-th electrode.
For a unique solution to the forward problem, one must specify the choice of ground,
\begin{equation}\label{ground}
    \sum_{l = 1}^{L} U_l = 0,
\end{equation}
and the current must satisfy Kirchhoff's Law:
\begin{equation}
    \sum_{l = 1}^{L}  I_l = 0.
\end{equation}

Denote the potential inside the domain $\Omega$ by $u$ or $v$ and the voltages on the boundary by $U$ or $V$. The variational formulation of the complete electrode model is given by 
\begin{equation}\label{var}
B_s((u, U), (v,V)) = \sum_{l =1}^L I_l \bar{V}_l,
\end{equation}
where  $v \in H^1(\Omega)$ and $ V \in \mathbb{C}^L$, and the sesquilinear form $B_s : H \times H \rightarrow \mathbb{C} $ is given by 
\begin{equation}
B_s((u, U), (v,V)) = \int_{\Omega} A \nabla u \cdot \nabla \bar{v}\, dx\, dy\, +\, \sum_{l =1}^L \frac{1}{z_l} \int_{e_l} (u -U_l) (v -\bar{V}_l) \, dS.
\end{equation}

Discretizing the variational problem leads to the finite element formulation. The domain $\Omega$ is discretized into small triangular elements with $N$ nodes in the mesh. Suppose $(u,U)$ is a solution to the complete electrode model with an orthonormal basis of current patterns $\varphi_k$. Then a finite dimensional approximation to the voltage distribution inside $\Omega$ is given by:
\begin{equation}\label{voltinsideapp}
u^{h} (z) = \sum_{k =1}^{N}  \alpha_k \varphi_k(z),
\end{equation}
and on the electrodes by
\begin{equation}\label{voltonelecapp}
U^{h} (z) = \sum_{k = N+1}^{N+(L-1)} \beta_{(k-N)} \vec{n}_{(k-n)},
\end{equation}
where discrete approximation is indicated by $h$ and the basis functions for the finite dimensional space $\mathcal{H} \subset H^{1} (\Omega)$ is given by $\varphi_{k}$, and $\alpha_{k}$ and $\beta_{(k-N)}$ are the coefficients to be determined.  Let
\begin{equation}
\vec{n}_{j} = (1,0,...,0,-1,0,...0)^T \in \R^{L \times 1},
\end{equation} 
where  $-1$ is in the $(j-1)$st position. 
The choice of $ \vec{n}_{(k-N)} $ satisfies the condition for a choice of ground in \eqref{ground},  since $\vec{n}_{(k-N)}$ in \eqref{voltonelecapp} results in 
\begin{equation}
U^{h}(z) = \Bigg ( \sum_{k =1}^{L-1} \beta_k, -\beta_1,... , -\beta_{L-1} \Bigg) ^T.
\end{equation}
In order to implement the FEM computationally we need to expand \eqref{var} using approximating functions \eqref{voltinsideapp} and \eqref{voltonelecapp} with $v =\varphi_{j}$ for $ j = 1,2,...N$ and $V = \vec{n}_{j} $ for $ j = N+1, N+2, ..... N+(L-1)$ to get a linear system 
\begin{equation}\label{linsys}
M \vec{b} = \vec{f},
\end{equation} 
where $\overrightarrow{b} =(\overrightarrow{\alpha},\overrightarrow{\beta})^T \in \mathbb{C}^{N+L-1} $ with the vector $\overrightarrow{\alpha} = (\alpha_1,\alpha_2,.....,\alpha_N) $ and the vector $\overrightarrow{\beta} = (\beta_1,\beta_2,.....\beta_{L-1}) $, and the matrix $G \in \mathbb{C}^{(N+L-1)} $ is of the form

\begin{equation}
M=\left ( \begin{array}{c c}
B & C\\
\tilde{C} & D\end{array} \right )
\end{equation} 

The right-hand-side vector is given by 
\begin{equation}
\overrightarrow{f} = ({\bf 0}, \tilde{I})^T,
\end{equation}
where ${\bf 0} \in \mathbb{C}^{1\times N} $ and $\tilde{I} = (I_1 -I_2, I_1 - I_3,....I_1-I_L) \in \mathbb{C}^{1\times (L-1)}  $. The entries of $\overrightarrow{\alpha} $ represent the voltages throughout the domain, while those of $\overrightarrow{\beta}$ are used to find the voltages on the elctrodes by
\begin{equation}
U^h = \mathcal{C} \overrightarrow{\beta}
\end{equation}
where $\mathcal{C} $ is the $L\times (L-1) $ matrix
\begin{equation}
\mathcal{C} = \left( \begin{matrix}
1 & 1 & 1 & \ldots & 1\\
-1 & 0 & 0 & \ldots & 0\\
0 & -1 & 0 & \ldots  & 0\\
&  &          \ddots          \\
0 & 0 & 0 & \ldots & -1 \end{matrix} \right).
\end{equation}

The entries of the block matrix $B$ are determined in each of the following cases:

\vspace{2mm}

\noindent $\bullet$ Case (i) {$\bf 1 \leq k,j \leq N$}. 

\vspace{2mm}
In this case $u^h \neq 0, U^h = 0, v \neq 0 $, but $V = 0 $. The sesquilinear form can be simplified to  
\begin{equation}
B_s((u^h, U^h), (v,V)) := \int_{\Omega} A \nabla u^h \cdot \nabla \bar{v} \, dx + \sum_{l =1}^L \frac{1}{z_l} \int_{e^l} u^h \bar{v} \, dS = 0.
\end{equation}

Thus, the $(k,j) $ entry of the block matrix $B $ becomes,
\begin{equation}\label{Bmat}
B_{kj} = \int_\Omega A \nabla \phi_k \cdot \nabla \overline{\phi}_j \,dx +\sum_{l=1}^{L} \frac{1}{z_l} \int_{e_l} \phi_k \overline{\phi}_j \, dS.
\end{equation}

The entries of the block matrix $C $ are determined as follows:

\vspace{2mm}

\noindent $\bullet$ Case (ii) {$ \bf 1 \leq k \leq N, N+1 \leq j \leq N+(L-1) $}. 

\vspace{2mm}
 
In this case $u^h \neq 0, U^h = 0, v =0 $, and $V \neq 0 $.  The sesquilinear form simplifies to  
\begin{equation}
B_s ((u^h, 0),(0,V)):= - \sum_{l = 1}^{L} \frac{1}{z_l} \int_{e^l} u^h \bar{V}_l \, dS  = I_1 - I_{j+1}
\end{equation}
Therefore, entries of C matrix becomes,
\begin{equation}\label{Cmat}
C_{kj} = - \bigg[ \frac{1}{z_l} \int_{e_l} \varphi_{k}(\overrightarrow{n}_{j})_{l} \, dS \bigg]
\end{equation}

\vspace{2mm}

\noindent $\bullet$ Case (iii) The entries of the block matrix $\tilde{C} $ are determined as follows: 
 
\vspace{2mm}

 For {$\bf N \leq k \leq N+(L-1), 1 \leq j \leq N$}. Here $u^h = 0, U^h \neq 0, v \neq 0, V =0 $. The expression for sesquilinear is 
\begin{equation}
B_s ((0,U^h),(v,0)):= - \sum_{l = 1}^{L} \frac{1}{z_l} \int_{e^l} U^h \bar{v}_l \, dS= 0
\end{equation}

Thus the {\it kj} entry of the $\tilde{C} $ is 
\begin{equation}
\tilde{C} = - \bigg[ \sum_{l =1}^{L} \int_{e_l} \overline{\varphi}_j \, dS - \frac{1}{z_l
+1} \int_{e_j +1} \overline{\varphi}_{j+1}\,  dS \bigg]
\end{equation}

\vspace{2mm}

\noindent $\bullet$ Case (iv) The entries of the block matrix $D $ are determined as follows:

\vspace{2mm}

For ${\bf N \leq k, j\leq N+(L-1)} $. Here $u^h =0, U^h \neq 0, v = 0, V \neq 0 $ The sequilinear form is given by
\begin{equation}
B_s ((0, U^h),(0,V)):=  \sum_{l = 1}^{L} \frac{1}{z_l} \int_{e^l} U^h \bar{V}_l \,  dS = I_1 - I_{j+1}
\end{equation}

Thus the entries of matrix D is given by 
\begin{equation}\label{dmat}
D_{kj} = \begin{cases}
      \frac{ \vert e_1 \vert}{z_1} + \frac{ \vert e_{j+1} \vert}{z_{j+1}}, &  j = k-N\\
 \frac{\vert e_1 \vert}{z_1},  & j \neq k-N.
\end{cases}
\end{equation}

Solving \eqref{linsys} gives us the coefficients $\beta_{(k-N)} $ required for the voltages $ U^h $ on the electrodes.

\subsection{Computing the quasi-conformal map} \label{subsec:quasi_conf}

A discrete approximation to the quasi-conformal map $\Phi^{(A_0)}$ was computed by solving the Beltrami equation \eqref{Belt} by the method referred to as {\em Scheme 1} in \cite{Gaidashev2008}. 
Let $T[h]$ denote the Hilbert transform of a function $h$
\begin{equation}
T[h] = \frac{i}{2\pi}\lim_{\epsilon\rightarrow 0}\int\int_{\C \setminus B(z,\epsilon)}\frac{h(\xi)}{(\xi - z)^2}d\bar{\xi}\wedge  \, d\xi,
\end{equation}
and $P[h]$ denote the Cauchy transform of  $h$
\begin{equation}
P[h] = \frac{i}{2\pi}\int\int_{\C}\frac{h(\xi)}{\xi - z}  \frac{h(\xi)}{\xi}-  d\bar{\xi}\wedge d\xi.
\end{equation}
 It is shown in \cite{Gaidashev2008} that a solution to equation \eqref{Belt}  can be computed as follows.
\begin{itemize}
\item[1.]  Begin with an initial guess $h^0$ to the solution of the equation
$$h^* = T[\mu h^*]+T[\mu]$$
\item[2.]  Compute the iterates
$$ h^{n+1} = T[\mu h^n]+T[\mu]$$
until the method converges to within a specified tolerance, and denote the solution by $h^*$.
\item[3.]  Compute the solution $f(z)$ from
$$f(z) = P[\mu(h^*+1)](z) + z$$
\end{itemize}
The iterates converge since the map $h \mapsto T[\mu h]$ is a contraction in $L_p(\C)$, $p>2$ by the extended version of the Ahlfors-Bers-Boyarskii theorem \cite{Gaidashev2008, Ahlfors1960}.
\subsection{Reconstruction}
Once we compute $w = f(z)$ for all discretized values of $z \in \Omega$, by using the reconstruction algorithm in \cite{Peter_2017} and the data from the forward modeling in subsection \ref{sec_forward}, we compute the isotropic conductivity $\tilde{\sigma}(w) = \tilde{\sigma}(f(z))$. Note that Calder\'on's method is a pointwise reconstruction algorithm. Then we get $\sigma(z) = \tilde{\sigma}(f(z))$. 

\section{Examples}
In this section we illustrate the method on a radially symmetric discontinuous conductivity on the unit disk with high and low contrast.  This simple example was chosen to illuminate the features of the reconstruction and facilitate comparison with the reconstructions by the D-bar method and Calder\'on's method in \cite{PropertiesPaper} of an isotropic conductivity of the same nature.

We consider four different conductivity tensors $A_0$ for the background on the unit disk $\Omega$:
\begin{align*}
A_0^1 = 
\left ( \begin{array}{c c}
1 & 0\\
0 & 1.3\end{array} \right), \quad 
A_0^2 = 
\left ( \begin{array}{c c}
1.3 & 0\\
0 & 1\end{array} \right), \quad 
A_0^3 = 
\left ( \begin{array}{c c}
1 & 0\\
0 & 4\end{array} \right), \quad 
A_0^4 = 
\left ( \begin{array}{c c}
4 & 0\\
0 & 1\end{array} \right).
\end{align*}
The tensor $A_0^1$ corresponds to $\mu = 0.0655$, $A_0^2$ corresponds to $\mu = -0.0655$,  $A_0^3$ corresponds to $\mu = 0.3333$, and $A_0^4$ corresponds to $\mu = -0.3333$, rounded to four digits after the decimal.  The images of  $\Omega$  under the quasi-conformal mapping computed by the method in Section \ref{subsec:quasi_conf} with initial guess $h^0= \mu{A_0}$ are found in Figure \ref{OmegaTilde}.
\begin{figure}[h!]  
\centering
\includegraphics[width=.35\textwidth]{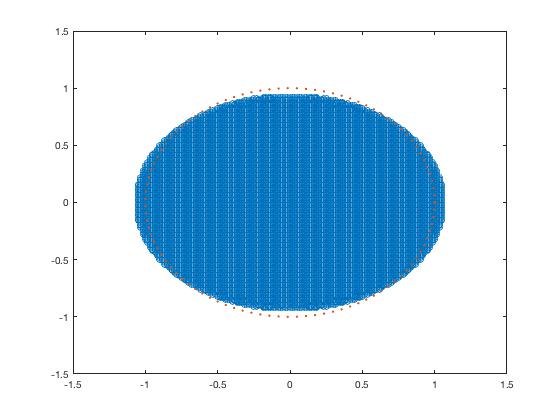} \hfil \includegraphics[width=.35\textwidth]{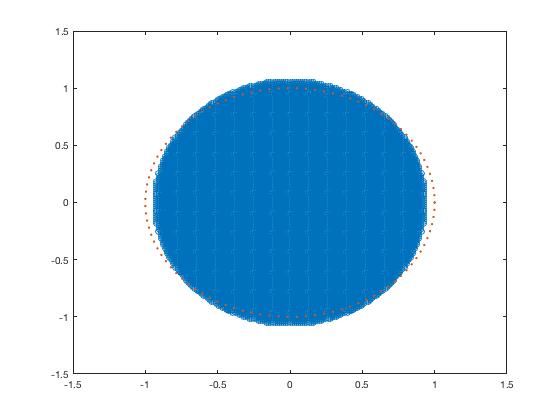} \\
\includegraphics[width=.35\textwidth]{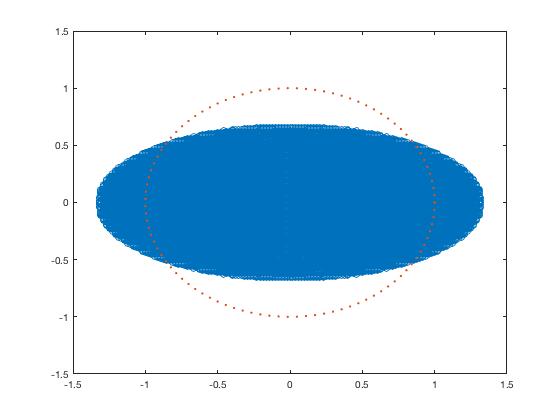} \hfil \includegraphics[width=.35\textwidth]{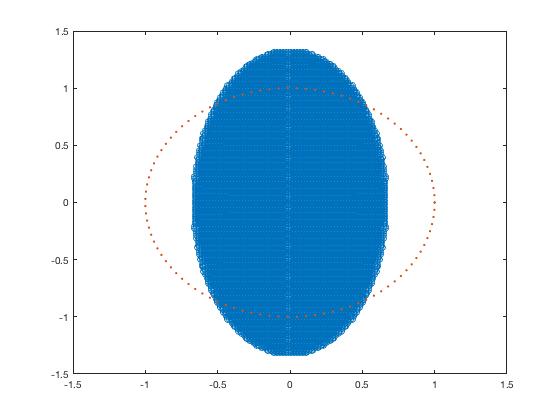}
\caption{The unit disk $\Omega$ and its image  $\tilde{\Omega}$ (in blue) under the computed quasi-conformal mapping $\Phi^{(A_0)}$ for $A_0^1$ (upper left), $A_0^2$ (upper right), $A_0^3$ (lower left), $A_0^4$ (lower right).}
\label{OmegaTilde}
\end{figure}

The discontinuous radially symmetric isotropic conductivity $\sigma$ is defined by 
$$\sigma_M(x) = \left\{
\begin{array}{ll}
 1, & 0.5<|x|\leq 1\\
M, & 0\leq |x| <0.5 
\end{array} \right.$$

Four anisotropic conductivity distributions were then constructed by defining $A_1(x) \equiv \sigma_{1.3}(x)A_0^1$, $A_2(x) \equiv \sigma_{1.3}(x)A_0^2$, $A_3(x) \equiv \sigma_{4}(x)A_0^3$, and $A_4(x) \equiv \sigma_{4}(x)A_0^4$.  Voltage data was simulated by the method described in Section \ref{sec:4} 
with  trigonometric current patterns defined by
\begin{equation}
T_\ell^k=\left\{
\begin{array}{ll}
\cos(k\theta_\ell), &  k=1,...,\frac{L}{2} \\
\sin\big((k-\frac{L}{2})\theta_\ell\big), &  k=\frac{L}{2}+1,...,L-1,
\end{array}\right.
\label{eq:currpatt}
\end{equation}
where $T_\ell^k$ specifies the current amplitude injected on electrode $\ell$ located at angular position $\theta_\ell$ for the $k$th pattern.

Reconstructions of $\sigma_{1.3}(x)$ with background tensor  $ A_0^1$ and $ A_0^2$ are found in Figure \ref{fig:recons_low_contrast}, and reconstructions of $\sigma_{4}(x)$ with background tensor  $ A_0^3$ and $ A_0^4$ are found in Figure \ref{fig:recons_high_contrast}.

\begin{figure}[H]  
\centering
\includegraphics[width=.35\textwidth]{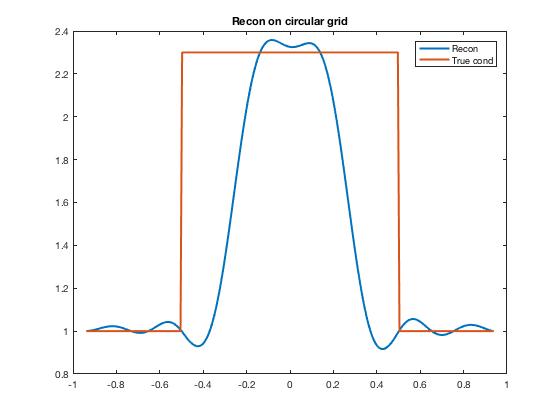} \hfil \includegraphics[width=.35\textwidth]{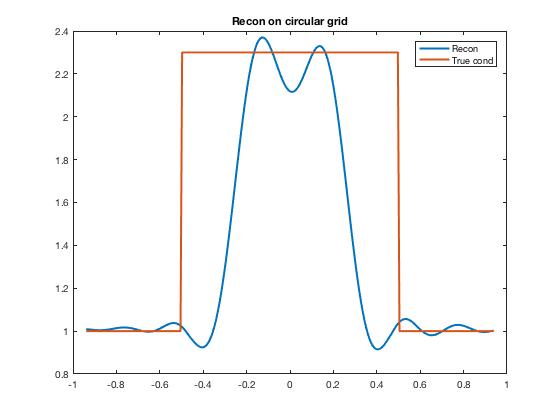} \\
\includegraphics[width=.35\textwidth]{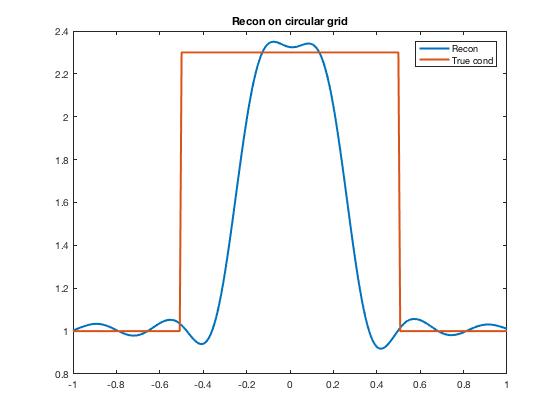} \hfil \includegraphics[width=.35\textwidth]{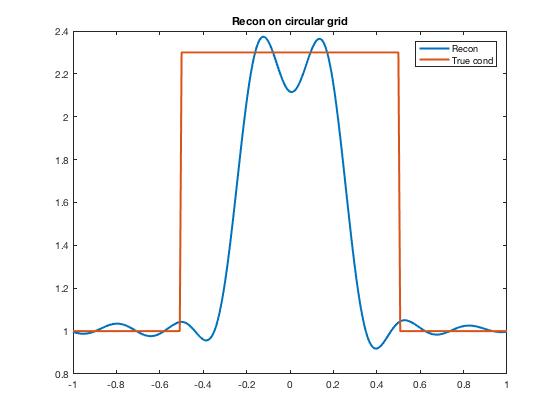}
\caption{Top row:  Cross sections along the x-axis of reconstructions of $\sigma_{1.3}(x)$ where $A_0 = A_0^1$ computed by Calder\'on's method for anisotropic conductivities  with truncation radius $R=1.8$ (left) and $R=2.0$ (right). Bottom row:  Cross sections along the x-axis of reconstructions of $\sigma_{1.3}(x)$ where $A_0 = A_0^2$ computed by Calder\'on's method for anisotropic conductivities with truncation radius $R=1.8$ (left) and $R=2.0$ (right).  }
\label{fig:recons_low_contrast}
\end{figure}

\begin{figure}[h!]  
\centering
\includegraphics[width=.35\textwidth]{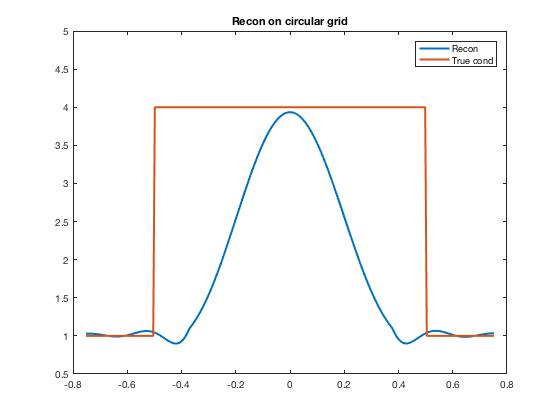} \hfil \includegraphics[width=.35\textwidth]{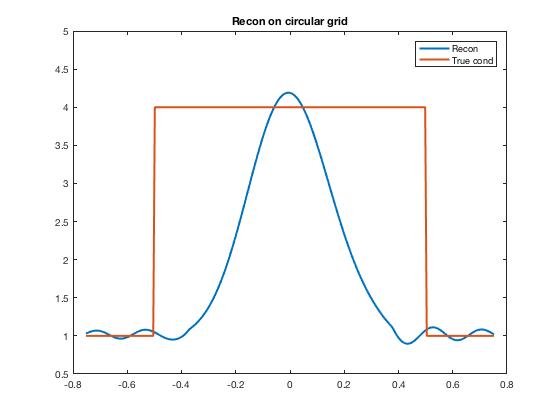} \\
\includegraphics[width=.35\textwidth]{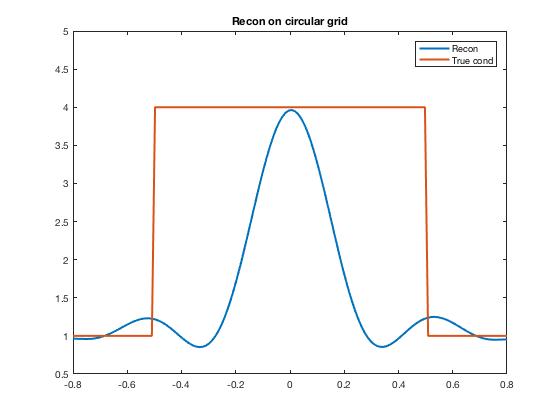} \hfil \includegraphics[width=.35\textwidth]{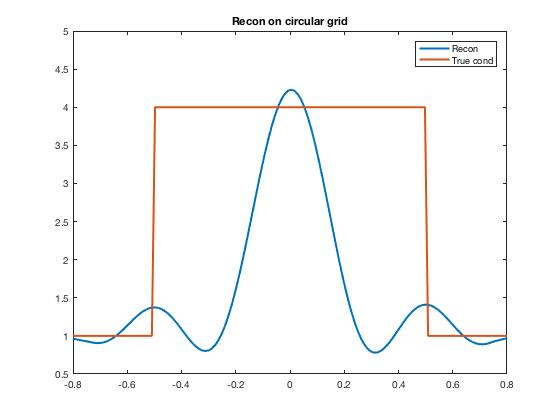}
\caption{Top row:  Cross sections along the x-axis of reconstructions of $\sigma_{4}(x)$ where $A_0 = A_0^3$ computed by Calder\'on's method for anisotropic conductivities with truncation radius $R=2.0$ (left) and $R=2.3$ (right). Bottom row:  Cross sections along the x-axis of reconstructions of $\sigma_{4}(x)$ where $A_0 = A_0^4$ computed by Calder\'on's method for anisotropic conductivities with truncation radius $R=2.0$ (left) and $R=2.3$ (right).  }

\label{fig:recons_high_contrast}
\end{figure}

%
\section{Conclusions}
A direct reconstruction algorithm for reconstructing the multiplicative scalar field for 2-D anisotropic conductivities with known entries for the background anisotropic tensors was presented based on Calder\'on's linearized method for isotropic conductivities.  The quasi-conformal map was used to prove injectivity of the linearized problem in the plane.  The map facilitates the reduction of the anisotropic problem to an isotropic problem that can then be solved by Calder\'on's method on the image of the original domain under the mapping, and pulled back to obtain the multiplicative scalar field.  The method is demonstrated on simple radially symmetric conductivities with jump discontinuity of high and low contrast.
Further work is needed to determine the method's practicality for more complicated conductivity distributions and experimental data.
Also, the method presented here is not applicable to the three-dimensional case.

\section*{Acknowledgment}

The project was supported by Award Number R21EB024683 from the National Institute Of Biomedical Imaging And Bioengineering.  The content is solely the responsibility of the authors and does not necessarily represent the official view of the National Institute Of Biomedical Imaging And Bioengineering or the National Institutes of Health. 
Y.-H. L.  is partially supported by the Ministry of Science and Technology Taiwan, under the Columbus Program: MOST-109-2636-M-009-006, 2020-2025.

\bibliographystyle{plain}
\bibliography{icebi_refs,ref,ref1,AnisotropyBib}

\end{document}